\documentclass[12pt,amsymb,fullpage]{amsart}
\usepackage{amssymb,amscd,pstricks}

\newtheorem{theorem}{Theorem}[section]
\newtheorem{defn}[theorem]{Definition}

\newtheorem{lemma}[theorem]{Lemma}

\newtheorem{eple}[theorem]{Example}
\newtheorem{rmk}[theorem]{Remarks}
\newtheorem{dsc}[theorem]{Discussion}
\newtheorem{nota}[theorem]{Notation}

\newsavebox{\indbin}
\savebox{\indbin}{\begin{picture}(0,0)
\newlength{\gnu}
\settowidth{\gnu}{$\smile$} \setlength{\unitlength}{.5\gnu}
\put(-1,-.65){$\smile$} \put(-.25,.1){$|$}
\end{picture}}

\newcommand{\be}{\begin{enumerate}}
\newcommand{\bd}{\begin{defn}}
\newcommand{\bt}{\begin{theorem}}
\newcommand{\bl}{\begin{lemma}}
\newcommand{\ee}{\end{enumerate}}
\newcommand{\ed}{\end{defn}}
\newcommand{\et}{\end{theorem}}
\newcommand{\el}{\end{lemma}}

\begin{document}
\title{Some Geometry of Nodal Curves}
\author{Tristram de Piro}
\address{55b, Via Ludovico Albertoni, Rome, 00152} \email{tristam.depiro@unicam.it}

\begin{abstract}
We find a geometrical method of analysing the singularities of a plane nodal curve. The main results will be used in a forthcoming paper on geometric Plucker formulas for such curves. Plane nodal curves, that is plane curves having at most nodes as singularities, form an important class of curves, as \emph{any} projective algebraic curve is birational to a plane nodal curve.
\end{abstract}

\maketitle

\begin{section}{An Analysis of the Nodes of an Algebraic Curve}

The purpose of this section is to develop the theory of plane nodal curves, using the Weierstrass preparation theorem. We use this theorem to analyse the nodes or ordinary double points of an algebraic curve. In this section, we will use the algebraic definition of a plane algebraic curve as defined by a single homogenous polynomial $F$ in the coordinates $\{X,Y,Z\}$ of $P^{2}$. Hence, such a curve may not be irreducible, or may be considered to have "non-reduced" factors. By a nodal curve, we mean any plane algebraic curve, having at most nodes as singularities, see the more precise statement below. It is extremely important to allow for non irreducible curves in the definition of a nodal curve. An example of such a curve is a union of $n$ lines in general position, that is no three of the lines intersect in a point.\\

We first restate a result from \cite{Newton}, in the special case of plane algebraic curves;\\

\begin{lemma}{Weierstrass Preparation for Plane Algebraic Curves}\\

Let $F(X,Y)$ be a polynomial in $L[X,Y]$ of the form;\\

$F(X,Y)=\sum a_{ij}X^{i}Y^{j}=0$\\

with $F(0,Y)\neq 0$ and $d=ord_{Y} F(0,Y)$. Then there exist unique elements $U(X,Y)$ and $G(X,Y)$ in $L[[X,Y]]$ such that;\\

$F(X,Y)=U(X,Y)G(X,Y)$\\

with $U(0,0)\neq 0$ and\\

$G(X,Y)=Y^{d}+c_{1}(X)Y^{d-1}+\ldots+c_{d}(X)$ $(*)$\\

with $c_{i}(X)\in L[[X]]$ and $c_{i}(0)=0$ for $1\leq i\leq d$.\\

\end{lemma}

We now give a characterisation of "nodes" for a plane algebraic curve (possibly not irreducible), which is a special case of this theorem. The following definition can be found in \cite{Ha};\\

\begin{defn}{"Node" or Ordinary Double Point of A Plane Algebraic Curve}\\

Let $F(X,Y)$ define a plane algebraic curve of degree $d$,with $F(0,0)=0$. We say that $(0,0)$ defines a "node" or ordinary double point of $F$, if $F=F_{2}+\ldots+F_{d}$, $F_{j}$ is a homogeneous polynomial of degree $j$ in $X$ and $Y$, for $2\leq j\leq d$, $F_{2}\neq 0$ and the linear factors of $F_{2}$ are \emph{distinct}.

\end{defn}

We then claim;\\

\begin{lemma}
Let $(0,0)$ define an ordinary double point of $F$, with the linear factors of $F_{2}$ given by $(aX+bY)$ and $(cX+dY)$. Then, if $l_{ab}$ defines the line $aX+bY=0$ and $l_{cd}$ defines the line $cX+dY=0$, we have that, for a line $l$ passing through $(0,0)$;\\

$I_{(0,0)}(F,l)=2$ iff $l$ is distinct from $l_{ab}$ and $l_{cd}$\\

Moreover, this condition characterises an ordinary double point. That is, if $C$ is a plane algebraic curve and $p\in C$ has the property that there exists exactly two distinct lines $\{l_{1},l_{2}\}$ passing through $p$ such that;\\

$I_{p}(C,l_{i})>2$ for $i\in\{1,2\}$\\

and\\

 $I_{p}(C,l)=2$\\

 for any other line $l$ passing through $p$, then $p$ defines an ordinary double point of $C$.

\end{lemma}

\begin{proof}

The proof is a straightforward algebraic calculation. For the first part of the lemma, suppose that $(0,0)$ defines an ordinary double point of $F$ and let $l$ be defined by $eX+fY=0$. Without loss of generality, assume that $e\neq 0$. We have to calculate;\\

$length({L[X,Y]\over <eX+fY,F(X,Y)>})=length({L[Y]\over <F(gY,Y)>})$, $(g={-f\over e})$\\

We have that;\\

 $F(gY,Y)=(ag+b)(cg+d)Y^{2}+O(Y^{3})$\\

Then, the result follows from the fact that $(ag+b)(cg+d)\neq 0$ iff $l$ is distinct from $l_{ab}$ and $l_{cd}$.\\

For the converse direction, let $F$ be the defining equation for $C$, and, without loss of generality, assume that $p$ corresponds to the origin $(0,0)$. By writing $F$ in the form $F=F_{1}+\ldots+F_{d}$, and using the same calculation as above, one deduces easily that $F_{1}=0$ and the polynomial $F_{2}$ splits into the distinct linear factors given by the equations of the lines $l_{1}$ and $l_{2}$.

\end{proof}

We apply this result to obtain;\\

\begin{lemma}
Let $(0,0)$ define an ordinary double point of $F$, such that the $Y$-axis is distinct from the tangent directions of the ordinary double point. Then, we can find $U(X,Y)$ and $G(X,Y)$, as in Lemma 1.1, such that $G$ has degree $2$ in $L((X))[Y]$.\\
\end{lemma}

\begin{proof}
By the assumption on the $Y$-axis and Lemma 1.3, we have that $ord_{Y}F(0,Y)=2$. Hence, the result follows immediately from Lemma 1.1.
\end{proof}

We then have;\\

\begin{lemma}
Let $G$ be given by the previous Lemma 1.4 and suppose that $char(L)\neq 2$, then we can find ${\eta_{1}(X),\eta_{2}(X)}$ in $L[[X]]$ such that;\\

$G(X,Y)=(Y-\eta_{1}(X))(Y-\eta_{2}(X))$ and $\eta_{1}'(0)\neq\eta_{2}'(0)$,\\ \indent \ \ \ \ \ \ \ \ \ \ \ \ \ \ \ \ \ \ \ \ \ \ \ \ \ \ \ \ \ \ \ \ \ \ \ \ \ \ \ \ \ \ \ \ \ \ \ \ \ \ \ $\eta_{1}(0)=\eta_{2}(0)=0$\\

as a formal identity in the ring $L[[X,Y]]$, where $\eta_{1}'(X),\eta_{2}'(X)$ denote the formal derivatives of $\eta_{1}(X)$ and $\eta_{2}(X)$ in $L[[X]]$, and the tangent directions of the node are given by $(Y-\eta_{1}'(0)X)$ and $(Y-\eta_{2}'(0)X)$.
\end{lemma}

\begin{proof}
We can write $G(X,Y)$ in the form;\\

$Y^{2}+c_{1}(X)Y+c_{2}(X)$ $(*)$\\

with $c_{i}(X)\in L[[X]]$ and $c_{i}(0)=0$ for $1\leq i\leq 2$, $(**)$. Suppose first that $G$ is reducible in $L((X))[Y]$. Then we have that;\\

$G(X,Y)=(Y-\eta_{1}(X))(Y-\eta_{2}(X))$\\

with $\{\eta_{1}(X),\eta_{2}(X)\}$ in $L((X))$. Substituting $\eta_{1}(X)$ in $(*)$, and using $(**)$, it follows immediately that $\eta_{1}(X)\in L[[X]]$
and $\eta_{1}(0)=0$. The same argument holds for $\eta_{2}(X)$ as well. Now, let $l_{ef}$ denote the line $eX+fY=0$, we may assume that $f\neq 0$ by the assumptions of Lemma 1.4. By a straightforward algebraic calculation, as above, we have that;\\

$I_{(0,0)}(F,l_{ef})=ord_{X}F(X,{-e\over f}X)$\\

where $ord_{X}$ may be taken either in $L[X]$ or $L[[X]]$. Now, using the fact that $U(X,Y)$, from Lemma 1.1, is a unit in $L[[X,Y]]$, we have that
$ord_{X}U(X,{-e\over f}X)=0$. By the elementary property of $ord_{X}$, that $ord_{X}(gh)=ord_{X}(g)+ord_{X}(h)$, for $g,h$ power series in $L[[X]]$, we must then have have;\\

$ord_{X}F(X,{-e\over f}X)=ord_{X}G(X,{-e\over f}X)$\\

\indent \ \ \ \ \ \ \ \ \ \ \ \ \ \ \ \ \ \ \ \ \ \ $=ord_{X}(eX+f\eta_{1}(X))(eX+f\eta_{2}(X))$ $(***)$\\

 It then follows from $(***)$, that;\\

$I_{(0,0)}(F,l_{ef})=2$ iff ${-e\over f}$ is distinct from $\{\eta_{1}'(0),\eta_{2}'(0)\}$.\\

By Lemma 1.3, the Definition 1.2 of an ordinary double point and the assumption in Lemma 1.4 on the tangent directions of the ordinary double point, we must then have that $\eta_{1}'(0)\neq \eta_{2}'(0)\neq 0$ and that the tangent directions are given by $(Y-\eta_{1}'(0)X)$ and $(Y-\eta_{2}'(0)X)$, as required. Now, suppose that $G$ is irreducible in $L((X))[Y]$, $(****)$, we will argue for a contradiction. Using the method of completing the square, which is valid with the assumption that $char(L)\neq 2$, $(****)$ can only occur if;\\

$Disc_{X}(G)=c_{1}(X)^{2}-4c_{2}(X)$ is \emph{not} a square in $L[[X]]$, $(\dag)$\\

Let $c_{1}(X)=X^{m}U(X)$ and $c_{2}(X)=X^{n}V(X)$, with $U(X)$ and $V(X)$ units in $L[[X]]$. Then, a straightforward algebraic calculation shows that;\\

 $(\dag)$ holds iff either $n<2m$ and $n$ is odd\\
  \indent \ \ \ \ \ \ \ \ \ \ \ \ \ \ \ \ \ \ \ or $n=2m$ and $U(X)-4V(X)$ is \emph{not} a square in $L[[X]]$\\

Let $G_{1}(X,Y)=G(X^{2},Y)=Y^{2}+c_{1}(X^{2})Y+c_{2}(X^{2})$.\\

We now claim that the discriminant;\\

$Disc_{X}(G_{1})=c_{1}(X^{2})^{2}-4c_{2}(X^{2})$ \emph{is} a square in $L[[X]]$, $(\dag\dag)$\\

In order to see this, first observe that $c_{1}(X^{2})=X^{2m}U(X^{2})$ and $c_{2}(X^{2})=X^{2n}V(X^{2})$, with $U(X^{2})$ and $V(X^{2})$ units in $L[[X]]$. Now, by the fact that $2n$ is even and $U(X^{2})-4V(X^{2})$ is always a square in $L[[X]]$, we see that the condition $(\dag)$ \emph{cannot} hold, hence $(\dag\dag)$ holds as required. Now, if $n<2m$ and $n$ is odd, we have that $2n<4m$ and;\\

$Disc_{X}(G_{1})=X^{2n}W(X)$ $(1)$\\

where $W(X)$ is the unit in $L[[X]]$ given by $X^{4m-2n}U(X^{2})^{2}-4V(X^{2})$. If $n=2m$ and $U(X)-4V(X)$ is not a square in $L[[X]]$, then $2n=4m$ and $ord_{X}(U(X)-4V(X))=r$ with $r$ odd. We then have that $ord_{X}(U(X^{2})-4V(X^{2})=2r$ and;\\

$Disc_{X}(G_{1})=X^{2n+2r}T(X)$ $(2)$\\

where $U(X^{2})-4V(X^{2})=X^{2r}T(X)$ and $T(X)$ is a unit in $L[[X]]$. In case $(1)$, we have that $n$ is odd, while in case $(2)$, we have that $n+r$ is odd. It follows that we can always find $s$ odd, a unit $R(X)$ in $L[[X]]$ and $Z(X)=X^{s}R(X)$ such that;\\

$Disc_{X}(G_{1})=Z(X)^{2}$\\

Let;\\

$\eta(X)={-c_{1}(X^{2})\over 2}+{Z(X)\over 2}$\\

By the fact that $R(X)=R(-X)$, we have $Z(X)=-Z(-X)$, hence, from the method of completing the square;\\

$G_{1}(X,Y)=G(X^{2},Y)=(Y-\eta(X))(Y-\eta(-X))$ $(\dag\dag\dag)$\\

It follows, from the construction of $\eta(X)$, that $\eta(0)=0$. Now, by making the formal substitution of $X^{1/2}$ for $X$ in $(\dag\dag\dag)$, we obtain that;\\

$G(X,Y)=(Y-\eta(X^{1/2}))(Y-\eta(-X^{1/2}))$ and $\eta(0)=0$, $(\dag\dag\dag\dag)$\\

as a formal identity in the ring $L[[X^{1/2},Y]]$. Now, by the fact that $\eta(0)=0$, we have;\\

$\eta(X)=a_{1}X+a_{2}X^{2}+a_{3}X^{3}+O(X^{4})$, $(\dag\dag\dag\dag\dag)$\\

Let $l$ be the line given by $Y-\lambda X=0$, for $\lambda\in L$. By a similar argument to the above, and using $(\dag\dag\dag\dag)$, we have that;\\

$I_{0,0}(F,l)=ord_{X}G(X,\lambda X)=ord_{X}(\lambda X-\eta(X^{1/2}))(\lambda X-\eta(-X^{1/2}))$\\

We have that;\\

$(\lambda X-\eta(X^{1/2}))(\lambda X-\eta(-X^{1/2}))=\lambda^{2}X^{2}-\lambda X[\eta(X^{1/2})+\eta(-X^{1/2})]+$\\
\indent \ \ \ \ \ \ \ \ \ \ \ \ \ \ \ \ \ \ \ \ \ \ \ \ \ \ \ \ \ \ \ \ \ \ \ \ \ \ \ \ \ \ \ \ \ \ $\eta(X^{1/2})\eta(-X^{1/2})$\\
\indent \ \ \ \ \ \ \ \ \ \ \ \ \ \ \ \ \ \ \ \ \ \ \ \ \ \ \ \ \ \ \ \ \ \ \ \ \ \ \ \ \ \ \ $=\lambda^{2}X^{2}-2\lambda X[a_{2}X+O(X^{2})]+$\\
\indent \ \ \ \ \ \ \ \ \ \ \ \ \ \ \ \ \ \ \ \ \ \ \ \ \ \ \ \ \ \ \ \ \ \ \ \ \ \ \ \ \ \ \ \ \ \ $[-a_{1}^{2}X+(a_{2}^{2}-2a_{1}a_{3})X^{2}+O(X^{3})]$\\

As $I_{0,0}(F,l)\geq 2$, we must have that $a_{1}=0$, and then;\\

$(\lambda X-\eta(X^{1/2}))(\lambda X-\eta(-X^{1/2}))=(\lambda-a_{2})^{2}X^{2}+O(X^{3})$\\

It then follows that $I_{0,0}(F,l)=2$ iff $\lambda\neq a_{2}$. In particular, this implies that there can only be \emph{one} tangent direction to the ordinary double point of $F$, given by $Y-a_{2}X=0$, which is a contradiction. This implies that $(****)$ cannot hold, hence the lemma is proved.

\end{proof}

\begin{rmk}
In Definition 6.3 of the paper \cite{Piero}, we defined a node of a plane algebraic curve to be the origin of two ordinary branches with distinct tangent directions, see that paper for relevant terminology, in particular, by a plane algebraic curve, we meant an \emph{irreducible} closed subvariety of $P^{2}$, having dimension $1$. This is slightly different from the definition that we have used here. For future reference and to avoid ambiguity, we will refer to Definition 1.2 as referring to an ordinary double point, reserving the terminology node for its use in \cite{Piero}. It follows immediately from the definition, that a node $p$ of a plane algebraic curve $C$ (irreducible) has the following property;\\

That there exist exactly $2$ distinct lines ${l_{1},l_{2}}$ passing through $p$ such that;\\

$I_{p}(C,l_{i})=3$, for $i\in \{1,2\}$\\

and, for any other line $l$, we have that;\\

$I_{p}(C,l)=2$ $(*)$\\

By Lemma 1.3, it follows that a node is an ordinary double point. However, the converse need not be true. In the sense of Definition 6.3 in \cite{Piero}, we can instead give the following geometric definition of an ordinary double point, for an irreducible plane algebraic curve $C$, as the origin of two \emph{linear} branches with distinct tangent directions, $(\dag)$. In the case of an irreducible plane algebraic curve $C$, this definition is equivalent to Definition 1.2, $(**)$. For, suppose that $p$ defines an ordinary double point in the sense of $(\dag)$, then it follows, see \cite{Piero};\\

That there exist exactly $2$ distinct lines ${l_{1},l_{2}}$ passing through $p$ such that;\\

$I_{p}(C,l_{i})>2$, for $i\in \{1,2\}$\\

and, for any other line $l$, we have that;\\

$I_{p}(C,l)=2$\\

By the same reasoning as above, Lemma 1.3, it follows that this implies $p$ corresponds to an ordinary double point in the sense of Definition 1.2. Conversely, suppose that $p$ does \emph{not} define an ordinary double point in the sense of $(\dag)$, then, we have one of the following cases;\\

Case 1. There is a single branch centred at $p$.\\

Case 2. There are at least three branches centred at $p$.\\

Case 3. There are two branches centred at $p$ and at least one of them\\
\indent \ \ \ \ \ \ \ \ \ \ \ is non-linear.\\

Case 4. There are two linear branches centred at $p$ with the same\\
\indent \ \ \ \ \ \ \ \ \ \ \ tangent directions.\\

Using Theorem 5.13 of \cite{Piero}, the property of tangent lines given in Theorem 6.2 of \cite{Piero} and the general result of the paper \cite{Newton}, Lemma 4.16, that Zariski multiplicity coincides with algebraic multiplicity, we have that, in Cases 1 and 4, there is a \emph{single} line $l_{1}$ with the property that $I_{p}(C,l_{1})>I_{p}(C,l)$ for any other line $l$ passing through $p$ while in Cases 2 and 3, we have that $I_{p}(C,l)\geq 3$, for \emph{any} line $l$ passing through $p$. By Lemma 1.3 again, it follows that, in all these cases, $p$ \emph{cannot} define an ordinary double point of $C$. Hence, $(**)$ is shown.\\

As an intuitive example, the figure "8", centred at the origin, may be considered to have an ordinary double point which is not a node. The reason being that the two branches centred at the origin are both inflexions, hence are linear but not ordinary.

\end{rmk}

We now extend the result of Theorem 2.10 in \cite{divisors} to the case of ordinary double points. We assume that $char(L)\neq 2$.

\begin{theorem}
Let $F(X,Y)=0$ define an irreducible plane algebraic curve $C$, with an ordinary double point at $(0,0)$. Let $(T,\eta_{1}(T))$ and $(T,\eta_{2}(T))$ be the power series representations of this point, as given by Lemma 1.5, and let $\{\gamma_{1},\gamma_{2}\}$ be the branches centred at $(0,0)$, see \cite{Piero}. Then, for any plane, possibly not irreducible, algebraic curve $H(X,Y)$ passing through $(0,0)$;\\

$H(T,\eta_{j}(T))\equiv 0$, for $j=1$ or $j=2$ iff $H$ contains $C$ as a\\
\indent \ \ \ \ \ \ \ \ \ \ \ \ \ \ \ \ \ \ \ \ \ \ \ \ \ \ \ \ \ \ \ \ \ \ \ \ \ \ \ \ \ \ \ \ \ \ \ \  component.\\

$ord_{T}H(T,\eta_{j}(T))=I_{\gamma_{j}}(C,H)$ otherwise, $j\in \{1,2\}$\\

where $I_{\gamma_{j}}$ denotes the branched intersection multiplicity, as defined in \cite{Piero} and \cite{divisors}.

\end{theorem}

\begin{proof}

The proof relies both on the method of Theorem 2.10 in \cite{divisors} and the methods of the paper \cite{Piero}. First, observe that we have $F(T,\eta_{j}(T))\equiv 0$ for $j\in \{1,2\}$, $(*)$. This follows immediately from Lemmas 1.4 and 1.5. If $H$ contains $C$ as a component, then, by the same argument as Theorem 2.10 in \cite{divisors}, using the Nullstellenstatz, we would have that $H(T,\eta_{j}(T))\equiv 0$, for $j=1$ \emph{and} $j=2$. For the converse direction, suppose that $H(T,\eta_{1}(T))\equiv 0$. By $(*)$, we have that $\eta_{1}(X)$ is an algebraic power series. Hence, we can interpret the equation $Y-\eta_{1}(X)$ as defining a curve $C_{1}$ on some etale extension $i:(A^{2}_{et},(00)^{lift})\rightarrow (A^{2},(00))$ such that $i(C_{1})\subset C$. As in Theorem 2.10 of \cite{divisors}, we can then argue to obtain that $H$ vanishes on $C$. The same argument holds if $H(T,\eta_{2}(T))\equiv 0$. For the second part of the theorem, we may therefore assume that $H$ has finite intersection with $C$ and $ord_{T}H(T,\eta_{j}(T))$ is finite for $j\in \{1,2\}$. Now, suppose that $deg(H)=e$ and let $\Sigma$ be a maximal linear system consisting of curves of degree $e$, having finite intersection with $C$. As in Theorem 2.10 of \cite{divisors}, we can write $H(X,Y)$ in the form $H(X,Y,\bar v_{0})$, where $\bar v\in Par_{\Sigma}$ and $F(X,Y)$ in the form $F(X,Y,\bar u_{0})$, for some non-varying constant $\bar u_{0}$. Similarly to Theorem 2.10 of \cite{divisors}, and using Lemma 1.5, we then have the sequence of maps;\\

$L[\bar v]\rightarrow {L[X,Y][\bar v]\over <F(X,Y,\bar u_{0}),H(X,Y,\bar v)>}\rightarrow {L[X]^{ext}[Y][\bar v]\over <(Y-\eta_{1}(X))(Y-\eta_{2}(X)),H(X,Y,\bar v)>}$\\

which corresponds to a sequence of finite covers;\\

$F_{1}\rightarrow F'\rightarrow Par_{\Sigma}$ $(1)$\\

We claim that the left hand morphism is etale at $(\bar v_{0},(00)^{lift})$, $(\dag)$. In order to see this, observe that the local rings ${L[X]^{ext}[Y]\over <F(X,Y,\bar u_{0})>}_{(00)}$ and ${L[X]^{ext}[Y]\over <(Y-\eta_{1}(X))(Y-\eta_{2}(X))>}_{(00)}$ are isomorphic, using the factorisations of Lemma 1.5, Lemma 1.1, and the invertibility of the unit $U$, obtained in Lemma 1.1, in the first local ring. It then follows that the completions of these local rings must be isomorphic as well. The claim $(\dag)$ then follows by the criteria for etale morphisms given in \cite{Mum}, (Theorem 3, p179). Intuitively, the power series factors $(Y-\eta_{1}(X))$ and $(Y-\eta_{2}(X))$ of $F(X,Y)$ together "preserve" the shape of the node at $(0,0)$.\\

We also have the maps;\\

${L[X]^{ext}[Y][\bar v]\over <(Y-\eta_{1}(X))(Y-\eta_{2}(X)),H(X,Y,\bar v)>}\rightarrow {L[X]^{ext}[Y][\bar v]\over <Y-\eta_{j}(X),H(X,Y,\bar v)>}$ ($j=1$ or $j=2$)\\

which correspond to inclusions;\\

$i_{2,j}:F_{2,j}\rightarrow F_{1}$, ($j=1$ or $j=2$) (2)\\

We will be interested in the covers $F_{2,j}\rightarrow Par_{\Sigma}$, obtained by combining $(1)$ and $(2)$. Let $d_{j}=ord_{X}H(X,\eta_{j}(X),\bar v_{0})$. We claim that the Zariski multiplicity of the cover $F_{2,j}\rightarrow Par_{\Sigma}$ at $(\bar v_{0},(00)^{lift})$ is $d_{j}$ as well, $(\dag\dag)$. This follows by imitating the corresponding proof in Theorem 2.10 of \cite{divisors}.\\

We now fix a non-singular model $C^{ns}\subset P^{w}$ of $C$, with birational presentation $\Phi_{\Sigma_{1}}:C^{ns}\rightarrow C$. Let $U_{\Phi_{\Sigma_{1}}}\subset C$ and $V_{\Phi_{\Sigma_{1}}}\subset C^{ns}$ be the canonical sets associated to this presentation, see \cite{Piero}. Corresponding to the family of forms $\{H_{\bar v}:\bar v\in Par_{\Sigma}\}$, we obtain a lifted family of forms $\{\overline{H_{\bar v}}:\bar v\in Par_{\Sigma}\}$ on $C^{ns}$. Let the branches $\{\gamma_{1},\gamma_{2}\}$ of the node $(0,0)$ of $C$ correspond to the distinct points $\{p_{1},p_{2}\}$ of $C^{ns}$. By the methods of \cite{Piero}, we may assume that $Base(\Sigma_{1})$ is disjoint from $\{p_{1},p_{2}\}$. From the definition of $\Sigma$, considered as a linear system on $C^{ns}$, we may also assume that $Base(\Sigma)$ is disjoint from $\{p_{1},p_{2}\}$. It then follows, from Definition 5.9 and Remarks 5.10 of the paper \cite{Piero}, that;\\

$I_{\gamma_{j}}(C,H)=Card(C^{ns}\cap\overline{H_{\bar v'}}\cap{\mathcal V}_{p_{j}})$, $\bar v'\in{\mathcal V}_{\bar v_{0}}$ generic in $Par_{\Sigma}$, $(**)$\\

Now define $F_{3}\subset Par_{\Sigma}\times P^{w}$ by;\\

$F_{3}(\bar v,x)$ iff $x\in (C^{ns}\cap {\overline H_{\bar v}})$\\

We have that $F_{3}\rightarrow Par_{\Sigma}$ is a finite cover and we may interpret the result $(**)$ by saying that this cover has Zariski multiplicity $d_{j}$ at $(\bar v_{0},p_{j})$, $(***)$.\\

Now let $C_{j}$ denote the irreducible curves defined by the algebraic power series $Y-\eta_{j}(X)$, for $j\in \{1,2\}$, and let $C_{12}$ be the reducible curve defined by the product $(Y-\eta_{1}(X))(Y-\eta_{2}(X))$. The curves $C_{j}$ are non-singular at $(0,0)^{lift}$, $(\dag\dag\dag)$, as one can see by calculating directly that the completions of the local rings ${L[X]^{ext}[Y]\over <Y-\eta_{j}(X)>}_{(00)}$ are in both cases equal to the formal power series ring $L[[X]]$. Let;\\

 $i_{j}:(C_{j},(00)^{lift})\rightarrow (C_{12},(00)^{lift})$, $j\in \{1,2\}$\\

 $\Psi:(C_{12},(00)^{lift})\rightarrow (C,(00))$\\

denote the inclusion morphisms and the locally etale morphism (at $(00)^{lift}$) defined respectively by the covers above. Let $W_{j}\subset C_{j}$ be the open sets defined by $(\Psi\circ i_{j})^{-1}(U_{\Phi_{\Sigma_{1}}})$. Then we obtain morphisms;\\

$\Theta_{j}=(\Phi_{\Sigma_{1}}^{-1}\circ\Psi\circ i_{j}):W_{j}\rightarrow C^{ns}$, $j\in \{1,2\}$\\

By $(\dag\dag\dag)$, the morphisms $\Theta_{j}$ extend to include the point $(0,0)^{lift}$ of $C_{j}$. We now show;\\

Claim 1. $\Theta_{j}((00)^{lift})\in \{p_{1},p_{2}\}, j\in\{1,2\}$\\

Claim 2. $\Theta_{1}((00)^{lift})\neq\Theta_{2}((00)^{lift})\ \ \ \ \ \ \ \ \ \ \ \ \  (****)$\\

Claim 3. $\Theta_{j}:(C_{j},(00)^{lift})\rightarrow (C^{ns},p_{j})$ is etale at $(00)^{lift}, j\in \{1,2\}$\\

\noindent Proof of Claim 1. Suppose, for contradiction, that $\Theta_{1}((00)^{lift})=p_{3}\notin \{p_{1},p_{2}\}$. Choose $x\in W_{1}\cap {\mathcal V}_{(00)^{lift}}$, then, by an elementary specialisation argument, $y=(\Psi\circ i_{1})(x)\in U_{\Phi_{\Sigma_{1}}}\cap {\mathcal V}_{(00)}$ and $\Phi_{\Sigma_{1}}^{-1}(y)=\Theta_{1}(x)\in C^{ns}\cap{\mathcal V}_{p_{3}}$. By elementary properties of specialisations, we would then have that $\Gamma_{\Phi_{\Sigma_{1}}}(p_{3},(00))$ in the correspondence between $C$ and $C^{ns}$, which is a contradiction. As the same argument holds for $\Theta_{2}$, the proof is shown.\\

\noindent Proof of Claim 2. Suppose, for contradiction, that $\Theta_{1}((00)^{lift})=\Theta_{2}((00)^{lift})=p_{1}$. Let $y\in V_{\Phi_{\Sigma_{1}}}\cap {\mathcal V}_{p_{2}}$, then $\Phi_{\Sigma_{1}}(y)\in U_{\Phi_{\Sigma_{1}}}\cap {\mathcal V}_{(00)}$. By Lemma 2.7 of \cite{Newton} (Lifting Lemma for etale covers), there exists a \emph{unique} $x\in C_{12}\cap {\mathcal V}_{(00)^{lift}}$ with $\Psi(x)=\Phi_{\Sigma_{1}}(y)$, hence, there clearly exists a unique $x'\in C_{j}\cap {\mathcal V}_{(00)^{lift}}$, with $(\Psi\circ i_{j})(x')=\Phi_{\Sigma_{1}}(y)$, for either $j=1$ or $j=2$. In either case, we would then have that $\Theta_{j}(x')=y$. By an elementary specialisation argument, this implies that $y\in C^{ns}\cap {\mathcal V}_{p_{1}}$. As the infinitesimal neighborhoods $C\cap {\mathcal V}_{p_{1}}$ and $C\cap {\mathcal V}_{p_{2}}$ are disjoint, this gives the required contradiction. As the same argument holds, reversing the roles of $p_{1}$ and $p_{2}$, the proof is shown.\\

\noindent Proof of Claim 3. We may assume that $\Theta_{1}((00)^{lift})=p_{1}$ and $\Theta_{2}((00)^{lift})=p_{2}$. Let $y\in V_{\Phi_{\Sigma_{1}}}\cap {\mathcal V}_{p_{1}}$, then, by a similar argument to the previous proof, we can find a \emph{unique} $x\in C_{1}\cap {\mathcal V}_{(00)^{lift}}$ with $\Theta_{1}(x)=y$. This implies that $\Theta_{1}$ is Zariski unramified at $(00)^{lift}$. By Theorems 2.7 and 2.8 of \cite{divisors} and the fact that $C^{ns}$ is smooth, if $\Theta_{1}$ fails to be etale, then it follows that it cannot be seperable either. In this case, the restriction of $\Theta_{1}$ to $W_{1}$ would also be inseperable, and, hence, either $i_{1},\Psi$ or $\Phi_{\Sigma_{1}}$ would be inseperable. As this is not the case, the proof is shown.\\

We have, therefore, shown $(****)$. Now observe that we can lift the family of forms $\{H_{\bar v}:\bar v\in Par_{\Sigma}\}$ to a family of forms on the etale cover $i:(A^{2}_{et},(00)^{lift})\rightarrow (A^{2},(00))$, which we will denote by $\{\overline{\overline{H_{\bar v}}}:\bar v\in Par_{\Sigma}\}$. We may then rewrite the cover $F_{2,j}\rightarrow Par_{\Sigma}$, using the more suggestive notation;\\

$F_{2,j}(\bar v,x)$ iff $x\in C_{j}\cap\overline{\overline{H_{\bar v}}}$\\

Moreover, observe that, if $x\in (W_{j}\cup (00)^{lift})\cap \overline{\overline{H_{\bar v}}}$, then $\Theta_{j}(x)\in C^{ns}\cap \overline{H_{\bar v}}$. Hence, restricting the covers if necessary, we can obtain a factorisation;\\

$(F_{2,j}, ((00)^{lift},\bar v_{0}))\rightarrow (F_{3},(p_{j},\bar v_{0}))\rightarrow (Par_{\Sigma},\bar v_{0})$\\

$(x,\bar v)\mapsto (\Theta_{j}(x),\bar v)\mapsto (\bar v)$\\

Using $(****)$, Claim 3, it is easy to check that the left hand cover is Zariski unramified at $((00)^{lift},\bar v_{0})$. It, therefore, follows that the Zariski multiplicity of the covers $F_{2,j}\rightarrow Par_{\Sigma}$ and $F_{3}\rightarrow Par_{\Sigma}$, at $((00)^{lift},\bar v_{0})$ and $(p_{j},\bar v_{0})$ respectively, is the same. Now, the result of the Theorem follows from $(\dag\dag)$ and $(***)$.\\

\end{proof}

\begin{rmk}

The geometric idea behind this proof is quite straightforward. The reader should have in mind the following hierarchy of images; a line, a cross, a node and a circle. The relationship between these images is simply expressed in many Gothic churches and cathedrals, in which a large circular window is placed above a series of "Gothic" nodal arches, The Abbazia di San Galgano in Italy is a particularly good example. In the language of Christianity, it expresses a relationship between the image of the Crucifixion and the image of The Lamp of Heaven. I hope to make this clearer in a book I am currently writing, entitled "Christian Geometry".\\

In more algebraic terms, this result may be expressed, by saying that the power series $(T,\eta_{j}(T))$, found in Lemma 1.5, define parametristions of the branches $\gamma_{j}$ of the node, in the sense of \cite{Piero}. It follows, from calculations in \cite{Plucker}, that, if another parametrisation (in the sense of \cite{Piero}) of the branch $\gamma_{j}$ is given, of the form $(T,\lambda_{j}(T))$, then $\eta_{j}(T)=\lambda_{j}(T)$. We will, therefore, refer to the power series, given by Lemma 1.5, as defining \emph{the} parametrisations of the ordinary double point (or node).

\end{rmk}

We now observe the following useful corollaries of Theorem 1.7;\\

\begin{lemma}
Let $F(X,Y)=0$ define an irreducible plane algebraic curve $C$, with an ordinary double point at $(0,0)$. Let $l_{\gamma_{1}}$ and $l_{\gamma_{2}}$ be the tangent lines to the two branches, centred at $(0,0)$, as defined in \cite{Piero}, and let $\eta_{1}(X)$ and $\eta_{2}(X)$ be the power series given by Lemma 1.5. Then, the equations of $l_{\gamma_{1}}$ and $l_{\gamma_{2}}$ are given by $(Y-\eta_{1}'(0)X)=0$ and $(Y-\eta_{2}'(0)X)=0$ respectively.

\end{lemma}

\begin{proof}

By Definition 6.3 of \cite{Piero}, the tangent lines $l_{\gamma_{j}}$ are characterised uniquely by the property that;\\

$I_{\gamma_{j}}(C,l_{\gamma_{j}})\geq 2$, $(j\in \{1,2\})$ $(1)$\\

By Theorem 1.7, we have that;\\

$I_{\gamma_{j}}(C,Y-\lambda X)=ord_{T}(\eta_{j}(T)-\lambda T)$ $(2)$\\

Combining $(1)$ and $(2)$, we then obtain immediately that the tangent $l_{\gamma_{j}}$ is given by $(Y-\eta_{j}'(0)X)=0$ as required.

\end{proof}

\begin{rmk}
This result is an improvement on Lemma 1.5, as this lemma does \emph{not} specify the correspondence between the power series $\{\eta_{1}(X),\eta_{2}(X)\}$ and the branches $\{\gamma_{1},\gamma_{2}\}$.

\end{rmk}

\begin{lemma}
Let hypotheses be as in the Theorem 1.7, with the additional assumption that $H(X,Y)$ is smooth at the point of intersection $(0,0)$ and has finite intersection with $C$. Let $l_{H}$ be the tangent line to $H$ at $(0,0)$ and let $l_{\gamma_{1}}$ and $l_{\gamma_{2}}$ be the tangent lines to the branches $\{\gamma_{1},\gamma_{2}\}$ of $C$. Then;\\

$I_{(0,0)}(C,H)=2$, if $l_{H}$ is distinct from $l_{\gamma_{1}}$ and $l_{\gamma_{2}}$\\

$I_{(0,0)}(C,H)>2$, otherwise.\\

Even without the assumption that $H(X,Y)$ is smooth at the point of intersection, we always have that;\\

$I_{(0,0)}(C,H)\geq 2$\\

\end{lemma}

\begin{proof}

By the main result of \cite{Newton}, Lemma 4.16, and Theorem 5.13 of \cite{Piero}, Branched Version of Bezout's Theorem, we have that;\\

$I_{(0,0)}(C,H)=I_{\gamma_{1}}(C,H)+I_{\gamma_{2}}(C,H)$ $(1)$\\

By Theorem 1.7, we have that;\\

$I_{\gamma_{j}}(C,H)=ord_{T}H(T,\eta_{j}(T))$, $j\in\{1,2\}$ $(2)$\\

where $\{\eta_{1}(T),\eta_{2}(T)\}$ are given by Lemma 1.5. By an application of the chain rule for differentiating algebraic power series, see the proof of Lemma 2.10 in \cite{Piero}, and the previous Lemma 1.9, we have that;\\

$ord_{T}H(T,\eta_{j}(T))>1$ iff $H_{X}|_{(0,0)}+H_{Y}|_{(0,0)}\eta_{j}'(0)=0$\\

\indent \ \ \ \ \ \ \ \ \ \ \ \ \ \ \ \ \ \ \ \ \ \ \ \ \ \ \ \ iff $dH_{(0,0)}\centerdot l_{\gamma_{j}}=0$. $(3)$\\

Now, the first part of the result follows immediately by combining $(1),(2)$ and $(3)$. The final part is clear, just using $(1)$.

\end{proof}

\begin{rmk}
It seems difficult to establish this type of result by purely algebraic methods, except in the simplest cases. In general, one would have to show that, for  polynomials of the form $F(X,Y)=(aX+bY)(cX+dY)+F_{1}(X,Y)$ and $H(X,Y)=(eX+fY)+H_{1}(X,Y)$, with $F_{1}$ and $H_{1}$ having first term in their homogeneous expansion of orders at least $3$ and $2$ respectively, and $\{l_{ab},l_{cd}\}$ distinct, that;\\

$length({L[X,Y]\over <F(X,Y),H(X,Y)>})_{(0,0)}>2$ iff $l_{ef}\notin \{l_{ab},l_{cd}\}$\\

I would be very interested to know how this can be done.

\end{rmk}

We finish this section by proving the following useful result concerning the effect of translations on ordinary double point (or nodes).\\

\begin{theorem}
Let $F(X,Y)=0$ define an irreducible algebraic curve $C$, with an ordinary double point at $(0,0)$. Let $l_{\gamma_{1}}$ and $l_{\gamma_{2}}$ be the tangent lines to the branches $\{\gamma_{1},\gamma_{2}\}$, centred at $(0,0)$, given in affine coordinates by $aX+bY=0$ and $cX+dY=0$. Let $(u,v)$ be \emph{any} choice of non-zero vector, with the property that the line $l$ defined by $vX-uY$ is \emph{distinct} from $l_{\gamma_{1}}$ and $l_{\gamma_{2}}$. Let $\{C_{t}\}_{t\in A^{1}}$ be the family of irreducible curves defined by;\\

$F_{t}(X,Y)=F(X-tu,Y-tv)=0$\indent  $(t\in A^{1})$\\

Then, for generic $t\in A^{1}\cap{\mathcal V_{0}}$, there exist exactly two points $\{q_{1},q_{2}\}=C\cap C_{t}\cap{\mathcal V}_{(0,0)}$. Moreover, in the terminology of \cite{Piero}, these points lie on the branches $\{\gamma_{1},\gamma_{2}\}$ respectively. Finally, the intersections are transverse.

\end{theorem}

\begin{rmk}

This property is a peculiar feature of nodal curves. For a general curve, with a smooth point at $(0,0)$, one would not expect to find \emph{any} such points of intersection, as is easily seen by direct calculation, in the simplest case of lines. The reader is strongly encouraged, by drawing a picture of a node, to see why, intuitively, the result should be true in this case. Our proof follows this intuitive idea. It seems clear geometrically that the result also holds for any deformation of a nodal curve $C$, which preserves the nodes, and for which the given node is not a base point of the deformation. Such deformations were studied extensively by Severi in \cite{Sev}.

\end{rmk}

\begin{proof}(Theorem 1.13)\\

By making a linear change of coordinates, we may assume that the line $l$ corresponds to the $Y$-axis, which is distinct from the tangent lines to the ordinary double point $(0,0)$. By Lemmas 1.4 and 1.5, we can find a factorisation;\\

$F(X,Y)=(Y-\eta_{1}(X))(Y-\eta_{2}(X))U(X,Y)$\\

as a formal identity in $L[[X,Y]]$, with $\eta_{1}(X)$ and $\eta_{2}(X)$ defining the parametrisations of the ordinary double point and $U(X,Y)$ a unit in $L[[X,Y]]$. We then have a corresponding factorisation of the translated curve;\\

$F_{t}(X,Y)=F(X,Y-t)=(Y-t-\eta_{1}(X))(Y-t-\eta_{2}(X))U(X,Y-t)$\\
\indent \ \ \ \ \ \ \ \ \ \ \ \ \ \ \ \ \ \ \ \ \ \ \ \ \ \ \ \ \ \ \ \ \ \ \ \ \ \ \ \ \ \ \ \ \ \ \ \ \ \ \ \ \ \ \ \ \ \ \ \ \ \ \ \ \ \ \ \ \ \ \ \ \ \ \ \ \ \ \ \ \ \  $(*)$\\

By the remarks at the beginning of Section 3 of \cite{Newton}, we can find an etale cover $i:(U,(00)^{lift})\rightarrow (A^{2},(00))$, such that the algebraic power series $\{\eta_{1}(X),\eta_{2}(X),U(X,Y)\}$ belong to the coordinate ring $R(U)$. Without loss of generality, we can assume that $U$ is irreducible. As $(Y-\eta_{1}(X))$ and $(Y-\eta_{2}(X))$ both vanish at $(00)^{lift}$, and are clearly irreducible in the power series ring $L[[X,Y]]$, they define irreducible algebraic curves $C_{1}$ and $C_{2}$ passing through $(00)^{lift}$. We can consider $Y-\eta_{j}(X)$ as defining a morphism from the algebraic variety $U$ to $A^{1}$. As $U$ is irreducible and $Y-\eta_{j}(X)$ is not identically zero in $R(U)$, the image of this morphism consists of an open subset $V_{j}\subset A^{1}$ containing $0$. By elementary dimension considerations, for $t\in V_{j}$, the corresponding fibre $(Y-t-\eta_{j}(X))=0$ defines an irreducible algebraic curve $C_{j,t}$ in $U$. We let $V=V_{1}\cap V_{2}$. It follows immediately that, for $t\in V$, the function $U(X,Y-t)$ also belongs to the fraction field $Frac(R(U))$. Moreover, if $S=\{t_{1},\ldots,t_{r}\}$ denotes the finitely many elements of $A^{1}$ for which the $Y$-axis intersects the algebraic curve $C$, then, a straightforward calculation, using $(*)$, shows that, for $t\in (V\setminus S)$, $(Y-t-\eta_{j}(X))|_{(00)^{lift}}\neq 0$ and, hence, $U(X,Y-t)|_{(00)^{lift}}\neq \{0,\infty\}$, $(\dag)$, in particular $U(X,Y-t)$ defines a unit in $L[[X,Y]]$. We let $V'=(V\setminus S)\cup \{0\}$. For $t\in V'$, let $R_{t}:=(U(X,Y-t)=\infty)$ be the infinite locus of $U(X,Y-t)$. By $(*)$ and the fact that the $G_{t}(X,Y)$ is finite on the affine plane $A^{2}$,  $R_{t}\subset C_{1,t}\cup C_{2,t}$. If $R_{t}$ is non-empty, by elementary dimension considerations, $R_{t}$ would contain at least one of the components $C_{j,t}$. Hence, $(00)^{lift}\in R_{t}$, which contradicts $(\dag)$. This shows that $R_{t}=\emptyset$ and $U(X,Y-t)$ belongs to $R(U)$ for $t\in V'$. The above calculation shows that the liftings $C_{t}^{lift}$ of the irreducible translated curves $C_{t}$ to $U$, for $t\in V'$, have the following decomposition;\\

$C_{t}^{lift}=C_{1,t}\cup C_{2,t}\cup W_{t}$ $(**)$\\

where $C_{1,t}$ and $C_{2,t}$ are the irreducible curves defined above, and $W_{t}$ is a (possibly empty) union of irreducible curves defined by $U(X,Y-t)=0$, disjoint from $(00)^{lift}$. We now show;\\

Claim 1. For $t\in V'\cap{\mathcal V}_{0}$, $q\in C\cap C_{t}\cap{\mathcal V}_{(0,0)}$ iff there exists a \emph{unique}\\
\indent \ \ \ \ \ \ \ \ \ \ \ \ \ $q^{lift}\in C^{lift}\cap C_{t}^{lift}\cap {\mathcal V}_{(0,0)^{lift}}$, with $i(q^{lift})=q$.\\

By Theorem 6.3 of \cite{deP} or even Lemma 2.7 of \cite{Newton}, the finite cover $(U/A^{2})$ (possibly localised) is Zariski unramified at $((00),(00)^{lift})$. Hence, if $q\in A^{2}\cap {\mathcal V}_{(0,0)}$, there exists a \emph{unique} $q^{lift}\in U\cap {\mathcal V}_{(0,0)^{lift}}$ with $i(q^{lift})=q$, $(***)$. In particular, if $q\in C\cap C_{t}\cap{\mathcal V}_{(0,0)}$, $q^{lift}$ is given by $(***)$, and as, by definition of $C^{lift}$ and $C_{t}^{lift}$, $q^{lift}\in C^{lift}\cap C_{t}^{lift}$, we have shown one direction of the claim. The other direction follows easily from definitions and the fact that, if $q^{lift}\in U\cap {\mathcal V}_{(0,0)^{lift}}$, then $i(q^{lift})\in A^{2}\cap{\mathcal V}_{(0,0)}$.\\

Claim 2. For $t\in (V'\setminus\{0\})\cap{\mathcal V}_{0}$;\\

$C^{lift}\cap C_{t}^{lift}\cap {\mathcal V}_{(0,0)^{lift}}=(C_{1}\cap C_{2}^{t}\cap {\mathcal V}_{(0,0)^{lift}})\cup (C_{2}\cap C_{1}^{t}\cap {\mathcal V}_{(0,0)^{lift}})$\\

We use the decomposition given in $(**)$. Suppose that $q^{lift}\in C^{lift}\cap C_{t}^{lift}\cap {\mathcal V}_{(0,0)^{lift}}$. First, we show that $q^{lift}$ cannot belong to $W_{0}$ or $W_{t}$. For, if either $W_{0}(q^{lift})$ or $W_{t}(q^{lift})$ holds, then, by specialisation, $W_{0}((00)^{lift})$ holds as well. This contradicts the fact that $W_{0}$ is disjoint from $(00)^{lift}$. The reader should compare the proof of Lemma 4.15 (Unit Removal) in \cite{Newton}, where a similar argument was used. Secondly, we show that $q^{lift}$ cannot belong to $C_{1}\cap C_{1,t}$ or $C_{2}\cap C_{2,t}$. This follows from an easy algebraic calculation. Namely, we would have that either the pair of functions $\{Y-\eta_{1}(X),Y-t-\eta_{1}(X)\}$ vanished at $q^{lift}$, or the pair of functions $\{Y-\eta_{2}(X),Y-t-\eta_{2}(X)\}$ vanished at $q^{lift}$. In either case, this implies the constant $t$ vanishes at $q^{lift}$, contradicting the assumption that $t\neq 0$. This shows the claim.\\

Claim 3. For $t\in (V'\setminus\{0\})\cap{\mathcal V}_{0}$;\\

 There exists a \emph{unique} $q_{1}^{lift}\in C_{1}\cap C_{2}^{t}\cap{\mathcal V}_{(0,0)^{lift}}$ and a \emph{unique}\\
  \indent $q_{2}^{lift}\in C_{2}\cap C_{1}^{t}\cap{\mathcal V}_{(0,0)^{lift}}.$\\

We show the first part of the claim, the proof of the second part is the same. The proof follows the methods of Section 2 in \cite{divisors}, which the reader is recommended to revise. We denote the coordinate ring of $V'$ by $L[t]_{h}$, for a polynomial $h(t)$ vanishing exactly at $(A^{1}\setminus V')$  We have the map;\\

$L[t]_{h}\rightarrow {L[X]^{ext}[Y][t]_{h}\over <Y-\eta_{1}(X),Y-t-\eta_{2}(X)>}$\\

which corresponds to a finite cover;\\

$F\rightarrow V'$, where $F\subset V'\times U$ is defined by $F(t,x)$ iff $x\in C_{1}\cap C_{2}^{t}$. We compute the Zariski multiplicity of the cover $F\rightarrow V'$ at $(0,(00)^{lift})$, $(\dag)$. First, observe that by Lemma 1.5, we have that;\\

$\eta_{1}(X)-\eta_{2}(X)=XU(X)$, for a unit $U(X)\in L[[X]]\cap L(X)^{alg}$\\

Hence, without loss of generality, it is sufficient to compute the Zariski multiplicity at $(0,0^{lift})$ of the cover $\phi$ determined by;\\

$L[t]\rightarrow {L[X]^{ext}[t]\over <XU(X)-t>}$ $(\dag\dag)$\\

By the inverse function theorem, or explicit calculation using the method of determining coefficients, we can find an algebraic power series $c(t)\in L[[t]]\cap L(t)^{alg}$, with $c(0)=0$ and $c'(0)\neq 0$, such that $c(t)U(c(t))=t$. We then have;\\

$XU(X)-t=XU(X)-c(t)U(c(t))$\\

\indent \ \ \ \ \ \ \ \ \ \ \ \ \ \ \ \ $=(X-c(t))U(X)+c(t)(U(X)-U(c(t)))$\\

\indent \ \ \ \ \ \ \ \ \ \ \ \ \ \ \ \ $=(X-c(t))(U(X)+c(t)V(X,t))$\\

\indent \ \ \ \ \ \ \ \ \ \ \ \ \ \ \ \ $=(X-c(t))W(X,t)$ for a unit $W(X,t)\in L[[X,t]]\cap L(X,t)^{alg}$\\

where, in the last step, we used the fact that $U(0)\neq 0$ and $c(0)=0$. We can now show directly that the cover $\phi$ determined by $(\dag\dag)$ is etale at $(0,0^{lift})$. This follows by observing that the map on formal power series;\\

${L[[X,t]]\over <(X-c(t))W(X,t)>}\rightarrow L[[t]]$, $f(X,t)\mapsto f(c(t),t)$\\

is an isomorphism, and applying the local criteria for etale morphisms, given in \cite{Mum} (p 179) (in this case $\phi$ induces an isomorphism on the formal power series ring $L[[t]]$ given by $\phi^{*}:t\mapsto t$) Then, one can use Theorems 2.7 and 2.8 of \cite{divisors}, to deduce that the cover $\phi$ determined by $(\dag\dag)$ is Zariski unramified at $(0,0^{lift})$. Hence, the claim follows.\\

We now complete the proof of the Theorem. By combining Claims 1,2 and 3, for generic $t\in A^{1}\cap{\mathcal V}_{0}$, (even more generally for $t\in {V'\setminus\{0\}}\cap{\mathcal V}_{0}$), the intersection $C\cap C_{t}\cap{\mathcal V}_{(00)}$ consists of at most two points $\{q_{1},q_{2}\}=\{i(q_{1}^{lift}),i(q_{2}^{lift})\}$, where;\\

 $q_{1}^{lift}=C_{1}\cap C_{2}^{t}\cap{\mathcal V}_{(0,0)^{lift}}$, $q_{2}^{lift}=C_{2}\cap C_{1}^{t}\cap{\mathcal V}_{(0,0)^{lift}}$\\

It is straightforward to see that $q_{1}^{lift}$ and $q_{2}^{lift}$ are distinct. If not, $q_{1}^{lift}=q_{2}^{lift}$ belongs to $C_{1}\cap C_{2}\cap {\mathcal V}_{(0,0)^{lift}}$, hence $q_{1}^{lift}=q_{2}^{lift}=(00)^{lift}$. This contradicts the fact we observed earlier, that $(00)^{lift}$ does not belong to $C_{1}^{t}$ or $C_{2}^{t}$, for $t\in {V'\setminus\{0\}}$. It follows that $q_{1}$ and $q_{2}$ are also distinct. If not, we would have that $Card(U\cap{\mathcal V}_{(00)^{lift}}\cap i^{-1}(q_{1}))=2$, contradicting the fact that the cover $(U/A^{2})$ is Zariski unramified at $((00),(00)^{lift})$. Hence, the intersection $C\cap C_{t}\cap{\mathcal V}_{(00)}$ consists of exactly two points $\{q_{1},q_{2}\}$, as required. In order to show that these points belong to the branches $\{\gamma_{1},\gamma_{2}\}$, we use the method of Theorem 1.7. Using the notation there, it is sufficient to check that the points $\{q_{1},q_{2}\}$ belong to the open sets $\{W_{1},W_{2}\}$ respectively and that the images $\{\Theta_{1}(q_{1}),\Theta_{2}(q_{2})\}$ belong to the infinitesimal neighborhoods $\{C^{ns}\cap{\mathcal V}_{p_{1}},C^{ns}\cap {\mathcal V}_{p_{2}}\}$ respectively. This is a straightforward exercise which we leave to the reader.\\

 Finally, we show the transversality result. It is clear that both the intersections $\{q_{1},q_{2}\}$ define nonsingular points of both $C$ and its translation $C_{t}$. It is, therefore, sufficient to show that the pairs of tangent lines $\{l_{q_{1},C},l_{q_{1},C_{t}}\}$ and $\{l_{q_{2},C},l_{q_{2},C_{t}}\}$ are distinct, $(\dag)$. The proofs of the remaining parts of the theorem show that the points $\{q_{1},q_{2}\}$ also lie on the translated branches $\{\gamma_{2}^{t},\gamma_{1}^{t}\}$ \emph{respectively} of $C_{t}$. It follows that we can find a pair $\{p_{2},p_{1}\}$, lying on the branches $\{\gamma_{2},\gamma_{1}\}$ of $C$ respectively, such that $l_{p_{2},C}$ is parallel to $l_{q_{2},C^{t}}$ and $l_{p_{1},C}$ is parallel to $l_{q_{1},C_{t}}$. In order to show $(\dag)$, it is, therefore, sufficient to prove that both the pairs $\{grad(l_{p_{2}}),grad(l_{q_{1}})\}$ and $\{grad(l_{p_{1}}),grad(l_{q_{2}})\}$ are distinct, $(\dag\dag)$. In order to show $(\dag\dag)$, we require the methods of \cite{Plucker}. We recall the definition of the gradient function, $grad$, see the remarks before Lemma 3.9 of \cite{Plucker}, given in the coordinate system $(X,Y)$ by;\\

$grad=-{F_{X}\over F_{Y}}$\\

It follows easily from the explanation in \cite{Plucker}, see specifically the power series calculation given immediately before Lemma 3.9 of \cite{Plucker}, that $grad$ defines a rational function on $C$ with the following property;\\

If $U\subset C$ denotes the open subset of nonsingular points of $C$ in finite position, whose tangent directions are \emph{not} parallel to the $y$-axis, then, for $x\in U$, $grad(x)$ is equal to the gradient of the tangent line $l_{x}$ in the coordinate system $(X,Y)$.\\

We may, without loss of generality, assume that the pairs $\{p_{2},q_{1}\}$ and $\{p_{1},q_{2}\}$ belongs to $U$. Hence, it is sufficient to show that $grad(p_{2})\neq grad(q_{1})$ and $grad(p_{1})\neq grad(q_{2})$, $(\dag\dag\dag)$. In order to show this last claim, fix a nonsingular model $C^{ns}$ of $C$, with birational morphism $\Phi:C^{ns}\leftrightsquigarrow C$, such that the branches $\{\gamma_{1},\gamma_{2}\}$ of the node centred at $(00)$ of $C$, correspond to infinitesimal neighborhoods $\{\mathcal V_{O_{1}},\mathcal V_{O_{2}}\}$  of $\{O_{1},O_{2}\}\subset C^{ns}$ in the fibre $\{O_{1},O_{2}\}=\Gamma_{[\Phi]}(y,(00))$, see Section 5 of \cite{Piero}. The function $grad$ lifts to a rational function $grad^{lift}=grad\circ\Phi$ on $C^{ns}$. Using the fact that $C^{ns}$ is nonsingular, it extends uniquely to a morphism $grad^{lift}:C^{ns}\rightarrow P^{1}$. We claim that $grad^{lift}(O_{1})$ defines the gradient of the tangent line $l_{\gamma_{1}}$ of the node, centred at $(00)$ of $C$, with a corresponding statement for $grad^{lift}(O_{2})$, $(\dag\dag\dag\dag)$. In order to see this, use Lemma 2.2 of \cite{Plucker}, to show that one can unambigiously assign a value $val_{\gamma_{1}}(grad)$ at the branch $\gamma_{1}$ of $C$. By the construction of $val_{\gamma}$, given before Lemma 2.1 of \cite{Plucker}, and the power series calculation, given before Lemma 3.9 of \cite{Plucker}, $val_{\gamma_{1}}(grad)$ gives the gradient of the tangent line $l_{\gamma_{1}}$. By Lemma 2.3 of \cite{Plucker}, which shows that $val_{\gamma}$ is birationally invariant, $val_{\gamma_{1}}(grad)=grad^{lift}(O_{1})$. Hence, $(\dag\dag\dag\dag)$ is shown. By the definition of a node, we obtain immediately that $grad^{lift}(O_{1})\neq grad^{lift}(O_{2})$. Now, using the fact that the pair $\{p_{2},q_{1}\}$ corresponds to points $\{p_{2}^{lift},q_{1}^{lift}\}$ in the infinitesimal neighborhoods $\{{\mathcal V}_{O_{2}},{\mathcal V}_{O_{1}}\}$ of $C^{ns}$ respectively, the result $(\dag\dag\dag)$ follows immediately from elementary properties of specialisations. This gives the result.

\end{proof}
\end{section}

\end{document}